\documentstyle[amssymb,amsfonts,12pt]{amsart}
\newtheorem{theorem}{Theorem}[section]
\newtheorem{lemma}[theorem]{Lemma}

\newtheorem{proposition}[theorem]{Proposition}

\theoremstyle{remark}

\def\QSet{\mbox{\rm\kern.24em
\vrule width.03em height1.48ex depth-.051ex \kern-.26em Q}}

\def\R{{\mathbb R}}
\def\A{{\mathbb A}}
\def\E{{\mathbb E}}

\def\Z{{\mathbb Z}}

\def\v{{\bf v}}\def\x{{\bf x}}\def\f{{\bf f}}\def\y{{\bf y}}

\def\be#1{\begin{equation}\label{#1}}

\def\bas{\begin{align*}}
\def\eas{\end{align*}}
\def\bi{\begin{itemize}}
\def\ei{\end{itemize}}
\newenvironment{proof}{\noindent {\bf Proof} }{\endprf\par}
\def \endprf{\hfill  {\vrule height6pt width6pt depth0pt}\medskip}
\def\emph#1{{\it #1}}

\begin{document}

\title[$L^2$ bounds for a Kakeya type maximal operator in $\R^3$]{$L^2$ bounds for a Kakeya type maximal operator in $\R^3$}

\author{Ciprian Demeter}
\address{Department of Mathematics, Indiana University, 831 East 3rd St., Bloomington IN 47405}
\email{demeterc@@indiana.edu}

\thanks{The  author is supported by a Sloan Research Fellowship and by NSF Grant DMS-0901208}
\begin{abstract}

We prove that the maximal operator obtained by taking averages at scale 1 along $N$ arbitrary directions on the sphere, is bounded in $L^2(\R^3)$ by $N^{1/4}{\log N}$. When the directions are  $N^{-1/2}$ separated, we improve the bound to  $N^{1/4}\sqrt{\log N}$. Apart from the logarithmic terms these bounds are optimal.
\end{abstract}
\maketitle

\section{Introduction}
Let $F:\R^d\to\R$ for some $d\ge 2$, and let $\Sigma\subset S^{d-1}$ be a collection of $N$ unit vectors. We will use the notation
$$\x=(x_1,x_2,\ldots,x_d),$$
$$\v=(v_1,v_2,\ldots,v_d).$$
 We will be concerned with the following operator
$$M_{0}^{\Sigma}F(\x)=\max_{\v\in\Sigma}|\int_{-1/2}^{1/2} F(\x+t\v)dt|.$$
Note that for each $\v$, Fubini's theorem implies that
$$M_{\v}F(\x)=\int_{-1/2}^{1/2} F(\x+t\v)dt$$
satisfies
$$\|M_{\v}F\|_1\lesssim\|F\|_1.$$ Thus, the triangle inequality and interpolation with the trivial $L^\infty$ bound proves that
$$\|M_0^\Sigma F\|_p\lesssim N^{1/p}\|F\|_p,$$
for each $1\le p<\infty$.
This estimate is  not optimal for $p>1$. The critical exponent is always $p=d$, and one expects an $O(N^{\epsilon})$ bound (or perhaps even a logarithmic bound) for $p\ge d$.

When $d=2$ this was confirmed in \cite{Wa}, where an $L^2$ bound of $\log N$ was proved. The optimal $L^2$ bound was shown in \cite{Ka} to be $\sqrt{\log N}$. Interpolation with  $L^\infty$ produces the $L^p$ bound $(\log N)^{1/p}$ for $2\le p<\infty$, which is known to be optimal (see \cite{Kai}).

In three dimensions, no nontrivial estimates seem to appear in the literature. The critical exponent is $p=3$ and it is very hard to deal with it directly.
Our theorem is concerned with $p=2$, where orthogonality methods are available.
\begin{theorem}
\label{main1}
Let $\Sigma\subset S^2$ be any collection of $N$ unit vectors.
For each $F\in L^2(\R^3)$ we have
$$\|M_0^\Sigma F\|_2\lesssim N^{1/4}\log N\|F\|_2.$$
\end{theorem}

In the separated case we have the following small improvement.

\begin{theorem}
\label{main}
Let $\Sigma\subset S^2$ be a collection of $N$ unit vectors such that
\begin{equation}
\label{equa35}
\|\v-\v'\|\gtrsim N^{-1/2}
\end{equation}
for each $\v\not=\v'\in\Sigma$.
Then for each $F\in L^2(\R^3)$ we have
$$\|M_0^\Sigma F\|_2\lesssim N^{1/4}\sqrt{\log N}\|F\|_2.$$
\end{theorem}
These results are sharp as far as the exponent of $N$ is concerned.  Indeed,  the function $F(\x)=\frac{1}{N^{1/2}\|\x\|^2}1_{N^{-1/2}<\|\x\|<2}$ has the $L^2$ norm  $\|F\|_2\sim N^{-1/4}$. On the other hand, since $$|F(\x)-F(\textbf{y})|\le \frac{1}{N\|\x\|^3}$$
for $\x,\textbf{y}$ in the support of $F$ which are separated by $O(N^{-1/2})$, if follows that $\|M_0^\Sigma F\|_2\sim 1$. It is not clear whether the logarithmic terms from the estimates in Theorems \ref{main1} and \ref{main} can be eliminated.

The  proofs  of Theorems \ref{main1} and \ref{main} rely on a wave packet decomposition similar to the one used in \cite{Dem}, where a different  proof is given to the two dimensional result in \cite{Ka}.
The annuli are first decoupled using the Chang-Wilson-Wolff inequality. As a result, $M_0^\Sigma$ is localized in frequency inside a fixed annulus. This is the source of a $\sqrt{\log N}$ loss in Theorems \ref{main1} and \ref{main}. The localized operator is then estimated using a few  analytic and combinatorial observations.

The advantage one has in  two dimensions, as explained in \cite{Dem}, is that the vector field $\v:\R^2\to\Sigma$ contributing to a fixed wave-packet is nicely localized inside a small arc on the circle. This allowed the splitting in \cite{Dem} of the wave packets into $\log N$ clusters, each having nice orthogonality properties. In $\R^3$, the vector field $\v$ is only loosely localized, inside a strip on the sphere. This creates difficulties in dealing with all annuli simultaneously, and motivates the initial decoupling.

A reproof of the two dimensional result along the same lines is sketched in Section \ref{two}.

\section{Relation with the Nikodym maximal function}
\label{Rel}
The result of  Theorem \ref{main} does not imply anything new about the Hausdorff dimension of the Kakeya sets in $\R^3$. In fact, it only implies (via standard considerations) that their dimension is at least 2. The best known lower bound is $\frac{5}{2}$ and it is due to Wolff \cite{Wol}. It is conjectured that the dimension should be 3.

While the fact that Kakeya sets in $\R^3$ have dimension at least 2 follows trivially from the fact that Kakeya sets in $\R^2$ have dimension  2, there does not seem to be a direct way of using the $L^2$ bound from \cite{Ka} for  $M_0^\Sigma$ in two dimensions,  to derive the three dimensional results in Theorems \ref{main1} and \ref{main}. However, as explained in Section \ref{general}, a variant of the two dimensional result will be used as part of the proof of Theorem \ref{main1}.

Define for each $0<\delta\ll 1$ and each integer $d\ge 2$
$$F_\delta^{**}:\R^d\to\R,\;\;\;\;\; F_\delta^{**}(\x)=\sup_{T}\frac1{|T|}\int_T|f|$$
where $T$ runs over all cylinders (tubes) in $\R^d$ containing $\x$, with length 1 and cross section radius $\delta$. This is sometimes referred to as the {\em Nikodym maximal function}.
When $d=2$, the optimal bound
\begin{equation}
\label{equa2}
\|F_\delta^{**}\|_{L^2(\R^2)}\lesssim \log (\frac1{\delta})^{1/2}\|F\|_{L^2(\R^2)}
\end{equation}
  was proved in \cite{Cor}\footnote{Actually, the bound in \cite{Cor} is for a larger operator, where averages are taken over tubes of eccentricity $\delta$, and arbitrary length}. The optimal bound
$\|F_\delta^{**}\|_{L^p(\R^2)}\lesssim \log (\frac1{\delta})^{1/p}\|F\|_{L^p(\R^2)}$
follows via interpolation with $L^\infty$. When $d=3$, the optimal (up to $\delta^{-\epsilon}$) operator norm  $\|F_\delta^{**}\|_{L^{5/2}(\R^3)\to L^{5/2}(\R^3)}$
was proved by Wolff \cite{Wol}.

There does not seem to exist a direct way of using \eqref{equa2}  to derive the  optimal bound from \cite{Ka}
$$\|M_0^\Sigma F\|_{L^2(\R^2)}\lesssim \log (\frac1{\delta})^{1/2}\|F\|_{L^2(\R^2)},$$
not even in the case when $\Sigma$ is a collection consisting of $\delta^{-1}$ unit vectors in $\R^2$, which are  $\delta$ separated. The same can be said when $d\ge 3$, too.
The averages on line segments appearing in the definition of $M_0^\Sigma$ are more "singular"; an additional smoothing effect appears when one averages over tubes. The contrast will be detailed in Section \ref{Que}.

On the other hand, any bound for $M_0^\Sigma$ in a given dimension $d\ge 2$ is easily seen to imply the same bound for $F_\delta^{**}$. We will briefly explain this below, in the case  when $d=3$ and $p=2$.
Let $\Sigma_\delta$ be a collection consisting of $\sim \delta^{-2}$ unit vectors in $\R^3$ such that
for each $\x\in \R^3\setminus \textbf{0}$ there exists $\v_\x\in \Sigma^{\delta}$ satisfying $\|\v_\x-\frac{\x}{\|\x\|}\|\le \delta$.
We show that Theorem \ref{main1} implies
$$\|F_\delta^{**}\|_{L^2(\R^3)}\lesssim \|M_0^{\Sigma_\delta}\|_{2\to 2}\|F\|_{L^2(\R^3)}.$$

Assume $F$ is positive.
Let $\x\mapsto T_\x$ be a measurable selection such that
\begin{enumerate}
\item $T_\x$ is a cylinder in $\R^3$ with length 1, cross section radius $10\delta$ and pointing in the direction $\v_\x$
\item the average of $F$ over $T_\x$ is greater than $\frac1{100}F_\delta^{**}(\x)$.
\end{enumerate}
Let $B_\x$ the ball of radius $10\delta$ centered at the same point as $T_\x$.
It will suffice to prove
$$\int_{\R^3}\left(\int_{T_\x}F\right )^2d\x\lesssim \delta^4\|M_0^{\Sigma_\delta}\|_{2\to 2}^2\int F^2$$
A simple geometric observation shows that
$$\int_{T_\x}F\lesssim \delta^{-1}\int_{B_\x}M_{\v_\x}F(\y)d\y\le \delta^{-1}\int_{B_\x}M_{0}^{\Sigma_\delta}F(\y)d\y.$$
Thus, via H\"older,
$$\int_{\R^3}\left(\int_{T_\x}F\right )^2d\x\lesssim \delta\int_{\R^3}\int_{B_\x}(M_{0}^{\Sigma_\delta}F(\textbf{y}))^2d\textbf{y}d\x$$
$$= \delta\int_{\R^3}(M_{0}^{\Sigma_\delta}F(\textbf{y}))^2[\int_{\R^3}1_{B_\x}(\textbf{y})d\x]d\textbf{y}\lesssim \delta^4\int_{\R^3}(M_{0}^{\Sigma_\delta}F(\textbf{y}))^2d\textbf{y}\lesssim \delta^4\|M_0^{\Sigma_\delta}\|_{2\to 2}\int_{\R^3} F^2$$

\section{$M_0$ restricted to an annulus: the separated case}
Fix $\Sigma$ as in Theorem \ref{main}. We will denote by $B$ the unit ball in $\R^3$. Let $\psi:\R\to\R$ be a positive Schwartz function such that $\hat{\psi}$ is supported in $[-1,1]$.

To prove Theorems \ref{main1} and \ref{main}, it suffices to prove the same bounds for
$$M_{0}F(x,y,z)=\max_{\v\in\Sigma}|T_{\v}(F)(x,y,z)|,$$
where
$$T_{\v}(F)(x,y,z)=\int_{\R} F(\x+t\v)\psi(t)dt.$$
For each $k$, we denote by $\A_k$ the annulus
$$\A_k:=\{2^{k-1}\le \|(\xi,\eta,\theta)\|\le 2^{k+1}\}.$$
Let $\mu\in C^{\infty}(\R^3)$ be supported in $\A_0$ so that
$$\sum_{k\in\Z}\mu_k(\xi,\eta,\theta):=\mu(2^{-k}\xi,2^{-k}\eta,2^{-k}\theta)=1,\;\;(\xi,\eta,\theta)\not=\textbf{0}.$$
Define
$$\widehat{S_kF}({\xi}, \eta, \theta)=\hat{F}({\xi}, \eta, \theta)\mu_k({\xi},\eta,\theta),$$
and note that 
$$F=\sum_{k}S_kF.$$

In this section we prove the main result leading to Theorem \ref{main}.

\begin{proposition}
\label{Prop1}
Let $\Sigma\subset S^2$ be a collection of $N$ vectors satisfying  \eqref{equa35}.
Then for each $k\ge 0$
$$\|M_0(S_kF)\|_2\lesssim N^{1/4}\|S_kF\|_2$$
\end{proposition}

Fix $k\ge 0$. Note that the Fourier transform of $S_kF$ is supported in the annulus $\A_k$.
 Let $\Xi_k$ be a partition of the sphere of radius $2^{k-1}$ into $\sim 2^{2k}$ caps $C_\omega$ with area $\sim 1$. The caps do not need to have the same shape. We will only insist  that $C_\omega \subset D(c_\omega,1)$, for some $c_\omega\in C_\omega$, where $D(c,1)$ is the disk centered at $c$ with radius $1$ on the sphere of radius $2^{k-1}$. For each $C_\omega$, let $\omega$ be the part of the cone with vertex at the origin, generated by $C_\omega$, which lies in the annulus $\A_k$,
$$\omega:=\{(\xi,\eta,\theta)\in \A_k:2^{k-1}\frac{(\xi,\eta,\theta)}{\|(\xi,\eta,\theta)\|}\in C_\omega\}.$$
We denote by $\Omega_k$ the collection of these tubes $\omega$.

Decompose
$$S_kF=\sum_{\omega\in\Omega_k}F_\omega,$$
where $\widehat{F_\omega}=\widehat{F}1_{\omega}.$ We will rely on the fact that $F_\omega$ are pairwise orthogonal.
For each $\omega\in\Omega_k$, let
$$A_\omega:=\{\textbf{w}\in S^2:|(\xi,\eta,\theta)\cdot \textbf{w}|\le 1\text{ for some }(\xi,\eta,\theta)\in\omega \}.$$
Then $A_\omega$ is a strip on the unit sphere of width $\sim 2^{-k}$.
Note that
$$T_{\v}(F_\omega)(x,y,z)=\int\hat{F_\omega}(\xi,\eta,\theta)\hat{\psi}(v_1\xi+v_2\eta+v_3\theta)e^{i(x\xi+y\eta+z\theta)}d\xi d\eta d\theta$$
is nonzero only if $\v\in A(\omega)$.

We need the following lemma (with $p=2$)
\begin{lemma}
\label{L543}If $\omega\in\Omega_k$ then
$$\|M_0F_\omega\|_p\lesssim \|F_\omega\|_p,$$
for each $1\le p\le \infty$.
\end{lemma}
\begin{proof}
Without loss of generality, assume $\omega$ points in the $z$ direction, in other words,  the $z$ axis intersects $C_\omega$. Then for each $\v\in A(\omega)$,
$$T_\v(F_\omega)(x,y,z)=\int \widehat{F_\omega}(\xi,\eta,\theta)m_\v(\xi,\eta,\theta)e^{i[{\xi}{x}+\eta y+\theta z]}d\xi d\eta d\theta,$$
where
$$m_\v(\xi,\eta,\theta)=\phi(\xi,\eta,\theta/2^k)\hat{\psi}({\xi}v_1+\eta v_2+\theta v_3),$$
and $\phi$ is an appropriate smooth bump function adapted to and supported on the ball of radius 10.
Note that $|v_3|\lesssim 2^{-k}$. This implies that for each $l_i\ge 0$,
$$\|\partial^{l_1}_\xi\partial^{l_2}_\eta\partial^{l_3}_\theta m_\v\|_{\infty}\lesssim 2^{-l_3k}$$
Thus the inverse Fourier transform $K_\v:=(m_\v)\check {\;}$ satisfies
$$K_\v(x,y,z)\lesssim K(x,y,z):=\frac{1}{(1+|x|)^{10}}\frac1{(1+|y|)^{10}}\frac{2^k}{(1+|z2^k|)^{10}},$$
uniformly over $\v\in A(\omega)$. Hence
$$\|\max_{\v\in \Sigma}|T_\v(F_\omega)|\|_p\le \|F_\omega*K\|_p\lesssim \|F_\omega\|_p\|K\|_1\lesssim \|F_\omega\|_p.$$
\end{proof}

\begin{proof}[of Proposition \ref{Prop1}]

We will rely on two estimates.
The first one is
$$\|M_{0}(\sum_{\omega\in\Omega_k}F_\omega)\|_2^2=\|\max_{\v\in\Sigma}|T_{\v}(\sum_{\omega\in\Omega_{\v,k}}F_\omega)|\|_2^2\le\sum_{\v\in\Sigma}\|T_{\v}(\sum_{\omega\in\Omega_{\v,k}}F_\omega)\|_2^2\lesssim $$$$ \sum_\v\|\sum_{\omega\in\Omega_{\v,k}}F_\omega\|_2^2=\sum_\v\sum_{\omega\in\Omega_{\v,k}}\|F_\omega\|_2^2=\sum_{\omega\in\Omega_k}\|F_\omega\|_2^2n(\omega),$$
where $\Omega_{\v,k}$ are those $\omega\in\Omega_k$ such that $\v\in A(\omega),$
and $n(\omega)$ is the cardinality of $\Sigma\cap A(\omega)$. The key is that, due to the relative uniform distribution of the vectors in $\Sigma$ on the unit sphere \eqref{equa35}, $A(\omega)$ contains at most $N2^{-k}$ vectors from $\Sigma$, if $2^k\lesssim {N^{1/2}}$, and at most $N^{1/2}$ such vectors if $2^k\ge N^{1/2}$.
Thus
\begin{equation}
\label{equa31}
\|M_0(S_kF)\|_2\lesssim \sqrt{\max\{N2^{-k},N^{1/2}\}}\|S_kF\|_2
\end{equation}
The second estimate is independent of the nature of $\Sigma$, and it relies on Lemma \ref{L543}, the Cauchy-Schwartz inequality, and the fact that for each  $\v\in\Sigma$, $T_\v(F_\omega)$ is nonzero only for $O(2^k)$ tubes $\omega$
$$
\int\max_{\v\in\Sigma}|T_\v(\sum_\omega F_\omega)|^2\le \int \max_{\v\in\Sigma}2^k\sum_\omega|T_\v( F_\omega)|^2\le$$$$  2^k\sum_\omega\int\max_{\v\in\Sigma}|T_\v( F_\omega)|^2\lesssim 2^{k}\int |S_kF|^2$$
Thus
\begin{equation}
\label{equa32}
\|M_0(S_kF)\|_2\lesssim 2^{k/2}\|S_kF\|_2.
\end{equation}
The proposition now follows from \eqref{equa31} and \eqref{equa32}.
\end{proof}

\section{$M_0$ restricted to an annulus: the general case}
\label{general}

In this section we will not impose any restrictions on $\Sigma$. We prove the main result leading to Theorem \ref{main1}.
\begin{proposition}
\label{Prop1dfjdfuio-}
For each $k\ge 0$,
$$\|M_0(S_kF)\|_2\lesssim N^{1/4}\sqrt{\log N}\|S_kF\|_2$$
\end{proposition}
The major obstruction in getting \eqref{equa31} comes from the fact that the strips $A(\omega)$ are now allowed to contain more than $\sqrt{N}$ vectors from $\Sigma$. Define
$$\Omega_k^{bad}:=\{\omega\in\Omega_k:|A(\omega)\cap\Sigma|\ge N^{1/2}\}$$
 Recall that each $A(\omega)$ is a strip on the unit sphere with width $\sim 2^{-k}$. When $k$ gets large enough compared to the smallest distance between points in $\Sigma$, these strips can be thought of as being lines. In this scenario, an application of the line incidence theorem of Szemer\'edi and Trotter \cite{ST}  shows that $\Omega_k^{bad}$ has $O(\sqrt{N})$ tubes. Lemma \ref{L543} combined with the orthogonality of $F_\omega$ would then  immediately prove the desired bound. Of course, the problem with this approach remains the fact that for small values of $k$, the strips $A(\omega)$ can not be equated with lines. In fact, it is very easy to see that there could be $\gg N^{1/2}$ tubes in $\Omega_k^{bad}$.

The new ingredient will be to use the following  variant of the two dimensional result from \cite{Ka}.
\begin{lemma}
\label{lemma98554}
Let $A$ be a strip of width $\sim 2^{-k}$, such as any $A(\omega)$, around a great circle $C$ on $S^2$. Then
$$\|\max_{\v\in\Sigma\cap A}|T_\v(S_kF)|\|_{L^2(\R^3)}\lesssim \sqrt{\log N}\|S_kF\|_{L^2(\R^3)}.$$
\end{lemma}
\begin{proof}
By rotation invariance we can assume  that $C$ lies in the plane $x=0$. Let $\f=(\xi,\eta,\theta)$.  For each $\v\in \Sigma\cap A$ let $\tilde{\v}\in C$ be such that $\|\tilde{\v}-\v\|\lesssim 2^{-k}$. Call $\tilde{\Sigma}$ the collection of all the  $\tilde{\v}$.  Note that $\tilde{\Sigma}$ has at most $N$ elements.
Recall that
$$T_\v(S_kF)(\x)=\int \widehat{F}(\f)\mu_k(\f)\widehat{\psi}(\f\cdot\v)e^{i\x\cdot\f}d\f.$$
It is easy to see that the inverse Fourier transform $K_\v$ of the multiplier $$m_\v(\f):=\mu_k(\f)\widehat{\psi}(\f\cdot\v)$$ satisfies
$$|K_\v(\x)|\lesssim \eta_{\tilde{\v}}(\x),$$
where $\eta_{\tilde{\v}}$ is obtained from the function
$$\eta_{(0,0,1)}(x,y,x):=\frac{2^{2k}}{(1+|x2^{k}|)^{10}(1+|y2^{k}|)^{10}}\frac{1}{(1+|z|)^{10}}$$
by applying any rotation that maps $(0,0,1)$ to $\tilde{\v}$. In other words, $\eta_{\tilde{\v}}$ is a smooth approximation to the characteristic function of the $2^{-k}\times 2^{-k}\times 1$ tube centered at the origin with the long side pointing in the direction $\tilde{\v}$.
We used the fact that $\eta_{\tilde{\v}}(\x)\sim \eta_{{\v}}(\x)$. The advantage now is that the vectors in $\tilde{\Sigma}$ are coplanar.
It suffices to prove that for each $F\in L^2(\R^3)$
$$\|\max_{\tilde{\v}\in\tilde{\Sigma}}|\int F(\x+\x')\eta_{\tilde{\v}}(\x')d\x'|\|_2\lesssim  \sqrt{\log N}\|F\|_2$$

Obtain similarly $\nu_{\tilde{\v}}(y,z)$
from
$$\nu_{(0,1)}(y,z):=\frac{2^{k}}{(1+|y2^{k}|)^{10}}\frac{1}{(1+|z|)^{10}}$$
by applying the rotation in the $x=0$ plane that maps $(0,0,1)$ to $\tilde{\v}$.

 Define the following two dimensional version of $M_0$
$$
M^{**}g(y,z)=\sup_{\tilde{\v}:=(0,\tilde{v}_2,\tilde{v}_3)\in\tilde{\Sigma}}|\int_{\R^2} g(y+y',z+z')\nu_{\tilde{\v}}(y',z')d y'dz'|
$$
We will need the fact that
\begin{equation}
\label{cdhjfh8437589798755}
\|M^{**}g\|_{L^2(\R^2)}\lesssim \sqrt{\log N}\|g\|_2.
\end{equation}
The proof of this will be postponed to Section \ref{two}.
We apply this to the functions
$$g(y,z)=F_{x}(y,z):=F(x,y,z)$$
 using H\"older's inequality to get
$$\|\max_{\tilde{\v}\in\tilde{\Sigma}}|\int_{\R^3} F(\x+\x')\eta_{\tilde{\v}}(\x')d\x'|\|_{L^2(\R^3)}\lesssim$$  $$\left(\int_\R\int_{\R} \|M^{**}F_{x+x'}(y,z)\|_{L_{y,z}^2(\R^2)}^2\frac{2^{k}}{(1+|x'2^{-k}|)^{10}}dx'dx\right)^{1/2}\lesssim
 \sqrt{\log N}\|F\|_{L^2(\R^3)}$$
\end{proof}

Next, run the following selection algorithm. Pick first any $\omega_1\in \Omega_k^{bad}$ such that $V_1:=A(\omega_1)\cap \Sigma$ has at least $\sqrt{N}$ elements. Then select $\omega_2\in \Omega_k^{bad}\setminus\{\omega_1\}$ such that
$V_2:=(A(\omega_2)\cap \Sigma)\setminus V_1$ has at least $\sqrt{N}$ elements. The algorithm stops when no such $\omega$ are available. In the end we will have the selected tubes $\omega_1,\dots,\omega_L$, and the pairwise disjoint sets $V,V_1,\ldots,V_L$ such that
$$\bigcup_{\omega\in\Omega_k}A(\omega)\cap\Sigma=\bigcup_{1\le l\le L}V_l\cup V,$$
and such that $V\cap A(\omega)$ has at most $\sqrt{N}$ elements for each $\omega\in\Omega_k$. Note that $L\le \sqrt{N}$.

Note first that the argument used to prove \eqref{equa31} also proves
\begin{equation}
\label{equaA1}
\|\max_{\v\in V}|T_\v(S_kF)|\|_2\lesssim N^{1/4}\|S_kF\|_2
\end{equation}
On the other hand Lemma \ref{lemma98554} implies that
\begin{equation}
\label{equaA1hjh878978}
\|\max_{\v\in \cup_{l}V_l}|T_\v(S_kF)|\|_2\le \sqrt{L}\max_{l}\|\max_{\v\in V_l}|T_\v(S_kF)|\|_2\lesssim N^{1/4}\sqrt{\log N}\|S_kF\|_2
\end{equation}

Proposition \ref{Prop1dfjdfuio-} now follows from \eqref{equaA1} and \eqref{equaA1hjh878978}.

\section{The decoupling of the annuli}

Define the following low, intermediate and high frequency restrictions 
$$\widehat{F^s}=\widehat{F}\sum_{k<0}\mu_k$$
$$\widehat{F^i}=\widehat{F}\sum_{3\log_2 N\ge k\ge 0}\mu_k$$
$$\widehat{F^l}=\widehat{F}\sum_{ k>3\log_2 N}\mu_k.$$
Recall the conditional expectation with respect to the $\sigma$- algebra  consisting of dyadic cubes of side length $2^{-j}$ in $\R^3$,
$$\E_jF(\x):=\sum_{Q:|Q|=2^{-3j}}\langle F,\frac{1_Q}{|Q|}\rangle 1_Q(\x)$$
 and let $\Delta_j=\E_{j+1}-\E_{j}$ be the martingale difference. Denote by
 $$\Delta(F)(\x)=(\sum_{j\ge 0}|\Delta_jF(\x)|^2)^{1/2}$$
 the discrete square function. We recall the following {\em good-lambda} inequality, which allows to compare  $F$ with its square function.

\begin{lemma}[The Chang-Wilson-Wolff inequality, \cite{CWW}]
\label{L5}

There exist constants $c_1,c_2$ such that for all $\lambda>0$ and $0<\epsilon<1$ one has
$$|\{\x\in\R^3:|F(\x)-\E_0F(\x)|>2\lambda,\Delta(f)(\x)<\epsilon\lambda\}|\le c_2e^{-\frac{c_1}{\epsilon^2}}|\{\x\in\R^3: \sup_{k\ge 0}|\E_kF(\x)|>\epsilon\lambda\}|.$$
\end{lemma}

Define $OF:=M_2(M^*F)$ where
$M_2g=(M^*(g^2))^{1/2}$
and $M^*g$ is the standard Hardy-Littlewood maximal function. The only  thing we need to know about these operators is that they are bounded on $L^2$.

Let $T$ be a linear bounded multiplier operator $T:L^2(\R^3)\to L^2(\R^3)$, that is
$$\widehat{TF}=m\hat{F},$$
for some $m\in L^{\infty}(\R^3)$. Define $T_kF:=T(S_kF)$.

The following two lemmas are proved in  \cite{Dem}. They are variants of similar lemmas from \cite{GHS}. The first result shows that $\Delta$ is dominated by a square function whose components are localized in frequency. We need frequency localization in order to be able to apply Proposition \ref{Prop1}.

\begin{lemma}\label{L1}
For each $T$ as above there exists $c_3>0$  such that for  each $F\in L^2(\R^3)$
$$\Delta(TF)(\x)\le c_3 (\sum_{k\in\Z}|O(T_kF)(\x)|^2)^{1/2}$$
almost everywhere.
\end{lemma}

The next Lemma controls error terms. The explanation for the extra $N$ in the denominator is that $\E_0$ and $T$ are almost orthogonal; $\E_0$ is "morally"  a Fourier restriction to the unit ball, while $T$ is restricted to frequencies larger than $N^2$.

\begin{lemma}
\label{L2}
Assume  that $m$ is zero on the ball with radius $N^2$.
Then there exists $c_4>0$   such that for  each $F\in L^2(\R^3)$
$$|\E_0(TF)(\x)|\le \frac{c_4}{N} (\sum_{k\in\Z}|O(T_kF)(\x)|^2)^{1/2}$$
almost everywhere.
\end{lemma}

\section{Proof of Theorem \ref{main1} and \ref{main}}

We will apply the lemmas from the previous section to
$$TF=T_{\v}(F):=\int F(\x+t\v)\psi(t)dt.$$
More exactly, we will distinguish three regimes. Write
$$T_{\v}=T_{\v}^s+T_\v^i+T_{\v}^l$$
where
$T_{\v}^sF=T_{\v}(F^s)$, $T_{\v}^iF=T_{\v}(F^i)$, $T_{\v}^lF=T_{\v}(F^l)$.

We prove Theorem \ref{main}, while Theorem \ref{main1} follows via trivial modifications.

\subsection{The small regime}

Let $\phi$ is a smooth bump function adapted to $2B$ such that
$$1_B\le \phi\le 1_{2B}.$$
Note that the inverse Fourier transform
$$K_\v(x,y,z):=(\phi(\xi,\eta,\theta)\hat{\psi}(v_1\xi+v_2\eta+v_3\theta))\check {\;}(x,y,z)$$
is  easily seen to satisfy
$$|K_\v(x,y,z)|\lesssim K(x,y,z):=(1+\|(x,y,z)\|)^{-10},$$
with bound independent of $\v$.

Observe that $\widehat{F^s}$ is supported in the unit ball. Thus we can write
$$M_0F^s(x,y,z)=\max_{\v\in \Sigma}|\int \hat{F^s}(\xi,\eta,\theta)\phi(\xi,\eta,\theta)\hat{\psi}(v_1\xi+v_2\eta+v_3\theta)e^{i(x\xi+y\eta+z\theta)}d\xi d\eta d\theta|$$$$\lesssim F^s*K(x,y,z).$$
 We conclude as before that
\begin{equation}
\label{equa7}
\|M_0F^s\|_p\lesssim \|F^s\|_p
\end{equation}
for each $1\le p\le \infty$. No restriction on $\Sigma$ (not even finiteness) was needed here.

\subsection{The intermediate regime}
It follows from Proposition \ref{Prop1}, followed by the triangle inequality and the almost orthogonality of $S_kF$ that
\begin{equation}
\label{equa8}
\|M_0F^i\|_2\lesssim \sum_{0\le k\le 3\log_2 N}N^{1/4}\|S_kF\|_2\lesssim N^{1/4}\sqrt{\log N}\|F\|_2
\end{equation}

\subsection{The large regime}
Let $\epsilon_N=\frac1{\sqrt{c_1\log N}}$. Define
$$G(F)=(\sum_{k> 3\log_2 N}O(\max_{\v\in\Sigma}|T_{\v}(S_kF)|)^2)^{1/2}.$$
From Lemma \ref{L1} and Lemma \ref{L2} we know that for each $\v\in S^2$
\begin{equation}
\label{equa21}
\Delta(T_{\v}^lF)(\x)\le c_3G(F)(\x)
\end{equation}
and
\begin{equation}
\label{equa22}
|\E_0(T_{\v}^lF)(\x)|\le \frac{c_4}{N}G(F)(\x)
\end{equation}
We dropped the low frequencies in the definition of $G(F)$, since $T_{\v}^lF$ is by definition localized at high frequencies.

For each $\lambda>0$,
$$\{\x:\max_{\v\in\Sigma}|T_{\v}^lF(\x)|>4\lambda\}\subset E_{\lambda,1}\cup E_{\lambda,2}\cup E_{\lambda,3},$$
where
$$E_{\lambda,1}=\{\x:\max_{\v\in\Sigma}|T_{\v}^lF(\x)-\E_0T_{\v}^lF(\x)|>2\lambda,G(F)(\x)\le \frac1{c_3}\epsilon_N\lambda\}$$
$$E_{\lambda,2}=\{\x:G(F)(\x)> \frac1{c_3}\epsilon_N\lambda\}$$
$$E_{\lambda,3}=\{\x:\max_{\v\in\Sigma}|\E_0T_{\v}^lF(\x)|>2\lambda\}.$$
By Lemma \ref{L5} applied to each function $T_{\v}^lF$ and using  \eqref{equa21} we get
$$|E_{\lambda,1}|\le \sum_{\v\in\Sigma}|\{\x:|T_{\v}^lF(\x)-\E_0T_{\v}^lF(\x)|>2\lambda,\Delta(T_{\v}^lF)(\x)\le \epsilon_N\lambda\}|\lesssim$$
$$\lesssim \frac{1}{N}\sum_{\v\in\Sigma}|\{\x:M^{*}(T_{\v}^lF)(\x)>\epsilon_N\lambda\}|.$$
Thus
\begin{equation}
\label{equa24}
\int_{0}^{\infty}\lambda|E_{\lambda,1}|d\lambda\lesssim \log N\|F^l\|_2^2
\end{equation}
The last inequality follows since each $T_{\v}$ is bounded on $L^2$.
Next, note that by Proposition \ref{Prop1} and the almost orthogonality of $S_kF$

$$\|G(F)\|_2\lesssim N^{1/4}\|F^l\|_2.$$
Thus
\begin{equation}
\label{equa25}
\int_0^\infty\lambda|E_{\lambda,2}|d\lambda\lesssim {N^{1/2}\log N}\|F^l\|_2^2\end{equation}
Finally, from \eqref{equa22} we have
\begin{equation}
\label{equa26}
\int_0^\infty \lambda|E_{\lambda,3}|d\lambda\lesssim \sum_{\v\in\Sigma}\int_0^\infty \lambda|\{\x:G(F)(\x)\gtrsim N\lambda\}|d\lambda\lesssim \frac1{N^{1/2}}\|F^l\|_2^2.
\end{equation}
We conclude that
$$\int_{0}^{\infty}\lambda\sum_{i=1}^3|E_{\lambda,i}| d\lambda\lesssim N^{1/2}\log N\|F\|_2^2$$
and thus
\begin{equation}
\label{equa7}
\|M_0F^l\|_2\lesssim N^{1/4}\sqrt{\log N}\|F^l\|_2
\end{equation}

A minute of reflection will show that no value of $\epsilon_N$ will be able to improve the bound $N^{1/4}\sqrt{\log N}$, via this type of argument.

\section{A two dimensional result}
\label{two}

We sketch a proof of \eqref{cdhjfh8437589798755}. Recall that $2^k\ge 0$. By splitting each $\nu_{\tilde{\v}}$ dyadically, and using the decay of its tail, it suffices to consider averages over $ 2^{m-k}\times 2^m$ tubes pointing in one of the directions from $\tilde{\Sigma}$, for some fixed $m\ge 0$. To be consistent with the previous analysis we can rescale and assume $m=0$. Let as before $\psi:\R\to\R$ be a positive Schwartz function such that $\hat{\psi}$ is supported in $[-1,1]$. Via an application of the Chang-Wilson-Wolff inequality as before, it will suffice to prove
\begin{equation}
\label{equa:84375rjgbfhhvg'}
\|\max_{{\v}\in\tilde{\Sigma}}|\int_{\R^2} G(\x+t\v+s\v^{\perp})\psi(t)2^k\psi(s2^k)dtds|\|_2\lesssim \|G\|_2
\end{equation}
if $\widehat{G}$ is supported in the unit ball, and also
\begin{equation}
\label{equa:84375rjgbfhhvg}
\|\max_{{\v}\in\tilde{\Sigma}}|\int_{\R^2} S_lF(\x+t\v+s\v^{\perp})\psi(t)2^k\psi(s2^k)dtds|\|_2\lesssim \|S_lF\|_2
\end{equation}
for each $l\ge 0$.

We will first prove \eqref{equa:84375rjgbfhhvg}. Decompose
$$S_lF=\sum_{T}F_T$$
where each $T$ is a an annular tube associated with an arc $C_T$ of length $\sim 2^{-l}$ on the unit circle
$$T:=\{(\xi,\eta)\in\A_l:\frac{(\xi,\eta)}{\|(\xi,\eta)\|}\in C_T\}.$$
Note that for each $T$
$$\int_{\R^2} F_T(\x+t\x+s\v^{\perp})\psi(t)2^k\psi(s2^k)dtds=$$$$\int_{\R^2}\widehat{F_T}(\xi,\eta)\widehat{\psi}(\v\cdot(\xi,\eta))\widehat\psi(2^{-k}\v^{\perp}\cdot(\xi,\eta))e^{i(x\xi+y\eta)}d\xi d\eta$$
is only nonzero if $\v$ belongs to the union of two arcs on the unit circle of length $\sim 2^{-l}$ lying orthogonal to $T$. The collection of these arcs will have bounded overlap, when $T$ varies over all possible tubes. On the other hand, an argument similar to the one from Lemma \ref{L543} will show that
$$\|\max_{{\v}\in\tilde{\Sigma}}|\int_{\R^2}F_T(\x+t\v+s\v^{\perp})\psi(t)2^k\psi(s2^k)dtds|\|_2\lesssim \|F_T\|_2$$
for each $T$. Indeed, assume $T$ points in the $\eta$ direction. Observe first that we can assume that $2^k\gtrsim 2^l$, otherwise the integrals are zero for each $\v$ (see the discussion in Section \ref{Que}). We can also assume $|\v_2|\lesssim 2^{-l}$. But then the multiplier
$$m_\v(\xi,\eta):=\phi(\xi,2^{-l}\eta)\widehat{\psi}(\v\cdot(\xi,\eta))\widehat\psi(2^{-k}\v^{\perp}\cdot(\xi,\eta))$$
satisfies
$$\|\partial^{a_1}_\xi\partial^{a_2}_\eta m_\v\|_{\infty}\lesssim 2^{-a_2l}$$
where $\phi$ is an appropriate smooth bump function adapted to and supported on the ball $10B$.
Thus $K_\v:=(m_\v)\check {\;}$ satisfies
$$|K_\v(x,y)|\lesssim \frac{1}{(1+|x|)^{10}}\frac{2^l}{(1+|y2^l|)^{10}}.$$
Combining these observations, \eqref{equa:84375rjgbfhhvg} follows.

To prove \eqref{equa:84375rjgbfhhvg'} we note that the multiplier
$$m_\v(\xi,\eta):=\phi(\xi,\eta)\widehat{\psi}(\v\cdot(\xi,\eta))\widehat\psi(2^{-k}\v^{\perp}\cdot(\xi,\eta))$$
satisfies for each $a_i\ge 0$
$$\|\partial^{a_1}_\xi\partial^{a_2}_\eta m_\v\|_{\infty}\lesssim 1$$
and hence
$$|(m_\v)\check {\;}(x,y)|\lesssim \frac{1}{(1+|x|)^{10}}\frac{1}{(1+|y|)^{10}}.$$

\section{Final remarks}
\label{Que}

The proofs of Theorem \ref{main1} and \ref{main} show the extra difficulty one encounters when dealing with $M_0$, as opposed to $F^{**}_\delta$. More precisely, $F^{**}_\delta$ is equivalent to the following smooth version
$$N^*(F)(x,y,z)=\max_{\v\in\Sigma_\delta}|N_\v(F)(x,y,z)|.$$
Here
$$N_\v(F)(\x):=\frac{1}{\delta^2}\int_{\R^3}F(\x+t\v+s\v^{\perp}+u\v^{\perp\perp})\psi(t)\psi(s/\delta)\psi(u/\delta)dtdsdu$$
 $\Sigma_\delta$ is a collection of $\delta^{-2}$, $\delta$ separated unit vectors,
and $\v^{\perp}, \v^{\perp\perp}$ are any two unit vectors such that $\v, \v^{\perp}, \v^{\perp\perp}$ are mutually orthogonal.
 Thus, the numerology relating the two operators is $N\sim \delta^{-2}$. Recall the notation $\f=(\xi,\eta,\theta)$. Note that
$$N_\v(F)(\x)=\int\hat{F}(\f)\hat{\psi}(\f\cdot\v)\hat{\psi}(\delta\f\cdot\v^\perp)\hat{\psi}(\delta\f\cdot\v^{\perp\perp})e^{i(\x\cdot\f)}d\f$$
is non zero only if $|\f\cdot\v|\le 1$, $|\f\cdot\v^\perp|\le \delta^{-1}$ and  $|\f\cdot\v^{\perp\perp}|\le \delta^{-1}$. It is easy to see that these can not simultaneously hold if $\|\f\|\ge 10\delta^{-1}$. In particular, each $N_\v$ -and thus $N^*$- only "see" the frequencies of $F$ smaller than $\delta^{-1}$. Moreover, for these small frequencies, the Fourier restriction $\hat{\psi}(\delta\f\cdot\v^\perp)\hat{\psi}(\delta\f\cdot\v^{\perp\perp})$ does not have any significant effect (it is essentially one), and thus $M_0(S_kF)\sim N^*(S_kF)$ whenever $2^{k}\lesssim \delta^{-1}$.
Thus
$$N^*(F)\lesssim M_0(F^s)+\sum_{0\le k\le \log (\delta^{-1})}M_0(F_k)$$
The bound
$$\|N^*(F)\|_{L^2(\R^3)}\lesssim \delta^{-1/2}\sqrt{\log (\delta^{-1})}$$
follows right away from Lemma \ref{L543} and Proposition \ref{Prop1}, without any appeal to a decoupling inequality.

It would be interesting if one could prove optimal results for $M_0$, for some $2<p<\infty$. The closer $p$ is to 3, the better are the implications on the dimension of the Kakeya sets.
 It is not clear whether the technology developed in \cite{Bo} or \cite{Wol} to prove such estimates for $N^*(F)$ could be used for $M_0$, too. One case of interest (and where some degree of orthogonality could still be exploited) is whether the estimate
$$\|M_0(F)\|_4\lesssim N^{\epsilon}\|F\|_4$$
holds. 

Another interesting question is whether one can prove similar $L^2(\R^3)$ bounds for the multi-scale maximal function
$$\max_{\v\in\Sigma}\sup_{\epsilon>0}\frac{1}{2\epsilon}\int_{-\epsilon}^{\epsilon}|F(\x+t\v)|dt.$$
Optimal bounds for all $p\ge 2$ in two dimensions were proved in \cite{Ka2}.


\begin{thebibliography}{99}\bibitem{Bo} J. Bourgain, {\em Besicovitch-type maximal operators and applications to Fourier
analysis} Geom. and Funct. Anal. 22 (1991), 147-187
\bibitem{CWW} S. Y. A. Chang, M. Wilson, T. Wolff, {\em Some weighted norm inequalities concerning the Schr\"odinger operator}, Comment. Math. Helv \textbf{60} (1985), 217-246.
\bibitem{Cor} A. Cordoba {\em The Kakeya maximal function and the spherical summation multipliers},  Amer. J. Math.  99  (1977), no. 1, 1-22
\bibitem{Dem} C. Demeter, {\em Singular integrals along $N$ directions in $\R^2$},  Proc.  Amer. Math. Soc.,  138  (2010),  no. 12, 4433-4442
\bibitem{GHS} L. Grafakos, P. Honzik, A. Seeger, {\em  On maximal functions for Mikhlin-H\"ormander multiplier}, Adv. Math. \textbf{204} (2006), no. 2, 363-378.
\bibitem{Ka} N. H. Katz, {\em Remarks on maximal operators over arbitrary sets of directions},  Bull. London Math. Soc.  31  (1999),  no. 6, 700-710
\bibitem{Ka2}  N. H. Katz, {\em Maximal operators over arbitrary sets of directions}, Duke Math. J. \textbf{97} (1999), no. 1, 67–79
\bibitem{Kai} U. Keich, {\em On $L^p$ bounds for Kakeya maximal functions and the Minkowski dimension in ${\bf R}^2$},  Bull. London Math. Soc.  31  (1999),  no. 2, 213–221.
\bibitem{ST} E. Szemer\'edi, W. T. Trotter, {\em Extremal problems in discrete geometry},  Combinatorica  3  (1983),  no. 3-4, 381-392.
\bibitem{Wa} S. Wainger, {\em Applications of Fourier transform to averages over lower dimensional sets},  Proc.
Sympos. Pure Math. 35 (1979) 85-94.
\bibitem{Wol} T. Wolff, {\em An improved bound for Kakeya type maximal functions}  Rev. Mat. Iberoamericana  11  (1995),  no. 3, 651-674.
\end{thebibliography}
\end{document}